\documentclass[ECP]{ejpecp} 
\usepackage{amsmath, amssymb, amsthm}
\usepackage{xcolor} 
\usepackage{mathtools}
\usepackage{bbm}
\usepackage{hyperref}
\usepackage{comment}
\newcommand{\norm}[1]{\left\lVert#1\right\rVert}
\newcommand{\abs}[1]{\left|#1\right|}

\SHORTTITLE{Restricted Quasiconvexity Isometry Property}

\TITLE{Restricted Quasiconvexity Isometry Property for Symmetric $\alpha$-Stable Random Matrices\
      } 

\AUTHORS{
  Sunder~Ram~Krishnan\footnote{Amrita Vishwavidyapeetham, Amritapuri, KL, India.
    \EMAIL{eeksunderram@gmail.com}}}


\AMSSUBJ{60B20; 60C05; 60D05; 60E07} 

\SUBMITTED{July 1, 2025} 
\ACCEPTED{December 13, 2014} 

\VOLUME{0}
\YEAR{2023}
\PAPERNUM{0}
\DOI{10.1214/YY-TN}

\ABSTRACT{We formulate a generalization of the Restricted Isometry Property (RIP) referred to as the Restricted Quasiconvexity Isometry Property (RQIP) for alpha stable random projections with $0<\alpha<1$. A lower bound on the number of rows for RQIP to hold for random matrices whose entries are drawn from a symmetric $\alpha$-stable ($S\alpha S$) distribution is derived. The proof leverages two key components: a concentration inequality for empirical fractional moments of $S\alpha S$ variables and a covering number bound for sparse $\ell_\alpha$ balls. The resulting sample complexity reflects the polynomial tail behavior of the concentration and reinforces an observation made in the literature that the RIP framework may have to be replaced with other sparse recovery formulations in practice, such as those based on the null space property.}

\begin{document}


\section{Introduction}

The Restricted Isometry Property (RIP) \cite{candes2005decoding} plays a foundational role in the theory of compressed sensing (CS) and high-dimensional signal recovery. It asserts that a matrix \( \Omega \in \mathbb{R}^{M \times N} \) approximately preserves the Euclidean norm of $k$-sparse vectors $x$ of length $N$ with at most $k$ non-zero entries (in other words, the norm $\|x\|_0\leq k$), that is,
\[
(1 - \delta) \|x\|_2^2 \leq \|\Omega x\|_2^2 \leq (1 + \delta) \|x\|_2^2,
\]
for all $k$-sparse $x \in \mathbb{R}^N$ with some $\delta>0$. This property ensures uniform stability and robustness, and has led to many deep results in sparse approximation, random matrix theory, and inverse problems; see \cite{fou} for an overview.
Many RIP results focus on matrices with sub-Gaussian entries or bounded orthonormal systems. In this setting, high-probability isometry guarantees are derived using tail bounds, chaining techniques, and volumetric covering arguments \cite{vershynin}. 

As regards generalized RIP definitions, Allen-Zhu et al. \cite{allenzhu} formulated a version for general $p$-norms and obtained almost tight bounds on the minimum number of measurements $M$ necessary for the RIP property to hold with high probability, for every $1\leq p <\infty$, with random binary matrices. Far less is known about measurement ensembles with heavy tails involving, for example, symmetric $\alpha$-stable ($S\alpha S$) distributions with $0 < \alpha < 2$. Previous work by Otero and Arce \cite{otero2011generalized}, which forms the main motivation for this work, addressed the case for $\alpha$-stable matrices with $1\leq \alpha \leq 2$, and established that at least $O(k \log(N/k))$ $\alpha$-stable measurements are needed for their RIP variant to hold with high probability. Again, a lower bound on the number of uniformly and independently subsampled rows from a bounded, orthonormal matrix necessary for RIP to hold was proved in Błasiok et al. \cite{jaro}. Recently, Chen et al. \cite{chen2024} established the RIP for block diagonal random matrices with elements from a class of generalized sub-Gaussian variables, thereby extending previously known results for the sub-Gaussian case. In an important work, Haviv and Regev \cite{Haviv2017} proved that RIP holds with probability close to 1 when rows are randomly sampled from a Fourier matrix. An improved lower bound for this last problem may be found in \cite{rao2019}.

Yet, RIP is not always necessary or desirable for practical recovery algorithms. Dirksen et al. \cite{dirksen2015} show a fundamental gap between RIP and the Null Space Property (NSP), which suggests that one can uniformly recover all $k$-sparse vectors in the optimal measurement regime $O(k \log(N/k))$ only within the latter framework. This distinction is central to the theory of CS and has been discussed, for instance, in later works such as \cite{cahill}.

We first define an isometry-type condition tailored to this setting, the so-called $(\delta,k)$-\emph{Restricted Quasiconvexity Isometry Property} (RQIP). This property, stated for a matrix $\Omega$ with $i$-th row $\omega_i^T$ and all $k$-sparse vectors $x$, is that
\[
(1 - \delta) C_{\alpha,p} (\gamma \|x\|_\alpha)^p \leq \frac{1}{M} \sum_{i=1}^M |\langle \omega_i, x \rangle|^p \leq (1 + \delta) C_{\alpha,p} (\gamma \|x\|_\alpha)^p
\]
holds for small enough $\delta>0$, with some $0<p<\alpha$ and constants $\gamma, C_{\alpha,p}>0$.

In this paper, we analyze matrices $\Omega$ with independent and identically distributed (i.i.d.)\ $S\alpha S(\gamma)$ entries, with $\gamma>0$ denoting the scale parameter, in the "ultra" heavy-tailed regime $\alpha \in (0,1)$. As with results in the standard literature on CS, our main theorem, henceforth referred to as the RQIP theorem, establishes that RQIP holds for $\Omega$ with high probability if the number of measurements $M$ exceeds a threshold, albeit with this threshold depending polynomially on $N$ (it also depends, obviously, on the target $\delta$). To be precise, our principal result shows that, as long as
\[M \geq \left( \frac{2 C_{\text{con}}}{\eta} C_{\text{net}}^k \left(\frac{3}{\delta}\right)^{k/p} \left(\frac{eN}{k}\right)^k \right)^{1/c_{\text{con}}}, \]
then $\Omega$ satisfies the $(\delta, k)$-RQIP with probability at least $1-\eta$, where $C_{\text{con}}$ depends only on $\alpha,p,$ and $\gamma$, $c_{\text{con}}=c_o\left(\frac{\alpha}{p}-1\right)$ with $0<c_0<1/2$, and $C_{\text{net}}=2(2^{1/\alpha-1})^2 + 2^{1/\alpha-1}$. The proof of this result relies on two ingredients: (i) A polynomial-decay concentration inequality for empirical fractional moments of $S\alpha S$ random variables, serving as a heavy-tailed analog of sub-Gaussian concentration \cite{vershynin}, and (ii) A covering number bound for the sparse $\ell_\alpha$ ball, valid for quasi-norms with $0 < \alpha < 1$, using entropy techniques adapted from classical arguments \cite{fou, vershynin}.
We combine these results using a standard $\varepsilon$-net argument, with careful choice of $\epsilon$ depending on the target RQIP deviation parameter $\delta$.

Our work provides the first isometry-type guarantee for matrices with extremely heavy-tailed entries, extending the scope of sparse recovery theory to the case $\alpha < 1$. The polynomial dependence on $N$, of the form $(\frac{N}{k})^{k/c_{\text{con}}}$, is exponentially larger than the $O(k \log(N/k))$ measurements sufficient to establish RIP and, consequently, prove rigorous sparse recovery guarantees using sub-Gaussian and $S\alpha S$ random measurement matrices with $1\leq \alpha \leq 2$, as noted above. This reinforces the idea that RIP-style conditions may be too rigid in highly heavy-tailed regimes, motivating the development of alternate guarantees (e.g., NSP \cite{dirksen2015} or robustness frameworks) for practical recovery algorithms.

\section{Preliminaries}
We start with a formal definition of the proposed RQIP and proceed to state and prove the two key ingredients in the proof of the RQIP theorem alluded to above: concentration and covering number bounds.
\begin{definition}[Restricted Quasiconvexity Isometry Property (RQIP)]
A matrix $\Omega \in \mathbb{R}^{M \times N}$ satisfies the $(\delta, k)$-RQIP for some $\delta \in (0,1)$ if, for all $k$-sparse vectors $x \in \mathbb{R}^N$, given $\alpha\in (0,1)$ and $\gamma>0$, there exists a constant $C_{\alpha,p} > 0$ (depending on $\alpha$ and $p$) such that:
$$ (1-\delta) C_{\alpha,p} (\gamma \norm{x}_\alpha)^p \le \frac{1}{M} \sum_{i=1}^M |\omega_i^T x|^p \le (1+\delta) C_{\alpha,p} (\gamma \norm{x}_\alpha)^p $$
for some $p \in (0, \alpha)$. Here, $\omega_i^T$ denotes the $i$-th row of $\Omega$, and $\norm{x}_\alpha = \left(\sum_{j=1}^N |x_j|^\alpha\right)^{1/\alpha}$ for $0 < \alpha < 1$.
\end{definition}

Let $\Omega$ be an $M \times N$ random matrix whose entries $\omega_{ij}$ are i.i.d. random variables, each drawn from the distribution $S\alpha S(\gamma)$, where $0 < \alpha < 1$ and $\gamma > 0$ is a scale parameter. For $y_i = \omega_i^T x$, if $x \in \mathbb{R}^N$ is a fixed vector and the components $\omega_{ij}$ of the vector $\omega$ are i.i.d. $S\alpha S(\gamma)$, then $y_i$ is an $S\alpha S(\gamma \norm{x}_\alpha)$ random variable. For $p \in (0, \alpha)$, it is well-known that the $p$-th absolute moment exists and is given by $\mathbb{E}[|y_i|^p] = C_{\alpha,p} (\gamma \norm{x}_\alpha)^p$, where $C_{\alpha,p} > 0$ is a constant depending only on $\alpha$ and $p$.

The proof of the main theorem, giving a lower bound on $M$ for $\Omega$ to satisfy RQIP with high probability, relies on the following two lemmas:

\begin{lemma}[Concentration of Empirical Fractional Moments for $S\alpha S$ Variables]
\label{lem:concentration}
Let $X_1, \dots, X_M$ be i.i.d. samples from the symmetric $\alpha$-stable distribution $S\alpha S(\gamma)$ with $0 < \alpha < 1$. Fix any $p \in (0, \alpha)$ and let $Y_i := |X_i|^p$.\\
Then there exist constants $C_{\text{con}}(\alpha,p,\gamma)$ and $c_{\text{con}} = c_0(\alpha/p-1)$ with $0<c_0<1/2$ such that, for any $\varepsilon >0$, there exists $M_0(\epsilon,\alpha,p,\gamma)>0$ and for $M>M_0$, we have
\[
\mathbb{P}\left( \left| \frac{1}{M} \sum_{i=1}^M Y_i - \mathbb{E}[Y_i] \right| > \varepsilon \right) \leq C_{\text{con}} \cdot M^{-c_{\text{con}}}.
\]
\end{lemma}


%

\begin{proof}
Let $Y_i = |X_i|^p$, where $X_i \sim S\alpha S(\gamma)$. Since $p < \alpha$, the $p$-th absolute moment $\mathbb{E}[Y_i] = \mathbb{E}[|X_i|^p]$ is finite. Let $\mu_Y = \mathbb{E}[Y_i]$. Our goal is to bound $\mathbb{P}\left( \left| \frac{1}{M} \sum_{i=1}^M Y_i - \mu_Y \right| > \varepsilon \right)$.

We employ a truncation strategy. Let $T > 0$ be a truncation threshold, to be chosen later. We decompose $Y_i$ into a bounded part and a tail part:
\[
Y_i^{\leq T} := Y_i \cdot \mathbbm{1}_{\{Y_i \leq T\}}, \quad Y_i^{> T} := Y_i \cdot \mathbbm{1}_{\{Y_i > T\}}.
\]
The deviation can then be decomposed as:
\[
\frac{1}{M} \sum_{i=1}^M Y_i - \mu_Y = \underbrace{\left( \frac{1}{M} \sum_{i=1}^M Y_i^{\leq T} - \mathbb{E}[Y_i^{\leq T}] \right)}_{\text{(I)}} + \underbrace{\left( \frac{1}{M} \sum_{i=1}^M Y_i^{> T} -  \mathbb{E}[Y_i^{> T}]\right)}_{\text{(II)}}.
\]
By the union bound, for any $\varepsilon > 0$,
\[
\mathbb{P}\left( \left| \frac{1}{M} \sum_{i=1}^M Y_i - \mu_Y \right| > \varepsilon \right) \leq \mathbb{P}\left( |\text{(I)}| > \frac{\varepsilon}{2} \right) + \mathbb{P}\left( |\text{(II)}| > \frac{\varepsilon}{2} \right).
\]

\subsubsection*{Bounding Term (I)}
Since $Y_i^{\leq T}$ are i.i.d. and bounded in $[0, T]$, we can apply Hoeffding's inequality:
\[
\mathbb{P}\left( \left| \frac{1}{M} \sum_{i=1}^M Y_i^{\leq T} - \mathbb{E}[Y_i^{\leq T}] \right| > \frac{\varepsilon}{2} \right) \leq 2 \exp\left( -\frac{M (\varepsilon/2)^2}{2 T^2} \right) = 2 \exp\left( -\frac{M \varepsilon^2}{8 T^2} \right).
\]

\subsubsection*{Bounding Term (II)}
This involves the tail behavior of $Y_i = |X_i|^p$. The tail probabilities of the random variables $X_i \sim S\alpha S(\gamma)$ are known to satisfy: 
\[
\mathbb{P}(|X_i| > t) \leq C_\alpha \gamma^\alpha t^{-\alpha} \quad \text{as } t \to \infty,
\]
for some constant $C_\alpha > 0$. We can bound $\mathbb{E}[Y_i^{> T}]$, for large enough $T$, using the tail probability:
\begin{align*}
\mathbb{E}[Y_i^{> T}] &= \mathbb{E}[|X_i|^p \cdot \mathbbm{1}_{\{|X_i|^p > T\}}] = \mathbb{E}[|X_i|^p \cdot \mathbbm{1}_{\{|X_i| > T^{1/p}\}}] \\
&= \int_0^\infty \mathbb{P}(|X_i|^p \cdot \mathbbm{1}_{\{|X_i|^p > T\}} > u) du \\
&= \int_T^\infty \mathbb{P}(|X_i|^p > u) du = \int_T^\infty \mathbb{P}(|X_i| > u^{1/p}) du \\
&\le \int_T^\infty C_\alpha \gamma^\alpha (u^{1/p})^{-\alpha} du = C_\alpha \gamma^\alpha \int_T^\infty u^{-\alpha/p} du.
\end{align*}
Since $p < \alpha$, we have $\alpha/p > 1$, so the integral converges:
\[ C_\alpha \gamma^\alpha \left[ \frac{u^{1-\alpha/p}}{1-\alpha/p} \right]_T^\infty = C_\alpha \gamma^\alpha \left( 0 - \frac{T^{1-\alpha/p}}{1-\alpha/p} \right) = \frac{C_\alpha \gamma^\alpha}{\alpha/p-1} T^{1-\alpha/p}. \]
Let $K_{\alpha,p,\gamma} = \frac{C_\alpha \gamma^\alpha}{\alpha/p-1}$. Then,
\[
\mathbb{E}[Y_i^{> T}] \leq K_{\alpha,p,\gamma} T^{1-\alpha/p}.
\]
For term (II), we use Markov's inequality:
\[
\mathbb{P}\left( \left|\frac{1}{M} \sum_{i=1}^M Y_i^{> T}-\mathbb{E}[Y_i^{> T}]\right| > \frac{\varepsilon}{2} \right) \leq \frac{4}{\varepsilon} \mathbb{E}[Y_i^{> T}] \leq \frac{4 K_{\alpha,p,\gamma}}{\varepsilon} T^{1-\alpha/p}.
\]

\subsubsection*{Choosing the Truncation Threshold $T$}
To balance the terms and achieve a polynomial decay, we choose $T$ such that the polynomial decay rate of the second term dominates, while term (I) decays faster.
For the polynomial decay, we need to choose $T$ to be a function of $M$ in such a way that $T$ is large when $M$ is so. Let $c_{\text{con}} = c_0(\alpha/p-1)$ be as specified in the lemma. We want term (II) to be roughly proportional to $M^{-c_{\text{con}}}$; choose $T$ such that $T^{1-\alpha/p}$ is of the order $M^{-c_0(\alpha/p-1)}$.\\
To be precise, define $T := \left( C' \cdot \frac{1}{\varepsilon} \right)^{1/(\alpha/p-1)} \cdot M^{c_0}$ for some constant $C'>0$ and $0<c_0<1/2$.
Then $T^{1-\alpha/p} = \left( C' \cdot \frac{1}{\varepsilon} \right)^{-1} \cdot M^{-c_0(\alpha/p-1)} = \frac{\varepsilon}{C'} M^{-c_{\text{con}}}$.
Thus,
\[
\mathbb{P}\left( \left|\frac{1}{M} \sum_{i=1}^M Y_i^{> T}-\mathbb{E}[Y_i^{> T}]\right| > \frac{\varepsilon}{2} \right) \leq \frac{4 K_{\alpha,p,\gamma}}{\varepsilon} \cdot \frac{\varepsilon}{C'} M^{-c_{\text{con}}} = \frac{4 K_{\alpha,p,\gamma}}{C'} M^{-c_{\text{con}}}.
\]
For sufficiently large $M$, this last bound is small since $c_{\text{con}} > 0$.

\subsubsection*{Verifying Term (I) with chosen $T$}
Now substitute $T$ into the bound for term (I):
\[
\mathbb{P}\left( \left| \frac{1}{M} \sum_{i=1}^M Y_i^{\leq T} - \mathbb{E}[Y_i^{\leq T}] \right| > \frac{\varepsilon}{2} \right) \leq 2 \exp\left( -\frac{M \varepsilon^2}{8 T^2} \right).
\]
Substituting $T^2 = \left( C' \cdot \frac{1}{\varepsilon} \right)^{2/(\alpha/p-1)} \cdot M^{2c_0}$, the last bound equals
\[
 2 \exp\left( -\frac{M \varepsilon^2}{8 \left( C'/\varepsilon \right)^{2/(\alpha/p-1)} M^{2c_0}} \right) = 2 \exp\left( - \frac{\varepsilon^{2 + 2/(\alpha/p-1)}}{8 (C')^{2/(\alpha/p-1)}} M^{1-2c_0} \right).
\]
Since $0 < c_0 < 1/2$, we have $1-2c_0 > 0$. This means the exponent $M^{1-2c_0}$ grows with $M$. Therefore, the exponential term decays faster than any polynomial in $M$, but note that the decay rate reduces as $\epsilon\rightarrow 0$.

\subsubsection*{Conclusion}
Combining the bounds above:
\begin{align*}
\mathbb{P}\left( \left| \frac{1}{M} \sum_{i=1}^M Y_i - \mu_Y \right| > \varepsilon \right) &\leq 2 \exp\left( - \frac{\varepsilon^{2 + 2/(\alpha/p-1)}}{8 (C')^{2/(\alpha/p-1)}} M^{1-2c_0} \right) + \frac{4 K_{\alpha,p,\gamma}}{C'} M^{-c_{\text{con}}}.\end{align*}
We may conclude, therefore, that there exists a constant $C_{\text{con}}(\alpha,p,\gamma)$ such that, for all large $M>M_0(\epsilon,\alpha,p,\gamma)$:
\[
\mathbb{P}\left( \left| \frac{1}{M} \sum_{i=1}^M Y_i - \mathbb{E}[Y_i] \right| > \varepsilon \right) \leq C_{\text{con}} \cdot M^{-c_{\text{con}}}.
\]
The constant $C_{\text{con}}$ depends on $\alpha, p, \gamma$ (through $K_{\alpha,p,\gamma}$). The constant $c_{\text{con}} = c_0(\alpha/p-1)$ is as stated in Lemma \ref{lem:concentration}.

\end{proof}

\begin{lemma}[Covering Number for Sparse $\ell_\alpha$ Ball]
\label{lem:covering}
Let $\mathbb{B}_{\alpha,k}^N := \{x \in \mathbb{R}^N : \norm{x}_\alpha \leq 1, \norm{x}_0 \leq k\}$ denote the sparse $\ell_\alpha$ unit ball, where $\norm{x}_\alpha = \left(\sum_{i=1}^N |x_i|^\alpha\right)^{1/\alpha}$ for $0 < \alpha < 1$, and $\norm{x}_0 = |\{i : x_i \neq 0\}|$ is the sparsity.
For any $0 < \epsilon < 1$, there exists an $\epsilon$-net of $\mathbb{B}_{\alpha,k}^N$ under the $\ell_\alpha$ quasi-norm of size at most $\left(\frac{C_{\text{net}}}{\epsilon}\right)^k {N \choose k}$, with constant $C_{\text{net}}=2(2^{1/\alpha-1})^2 + 2^{1/\alpha-1}$.
\end{lemma}

\begin{proof}
The proof proceeds in two main steps: first, we establish a bound for the covering number of the $k$-dimensional $\ell_\alpha$ unit ball, and then we extend this result to the $N$-dimensional sparse $\ell_\alpha$ ball using a combinatorial argument.

\textbf{Step 1: Covering the $k$-dimensional $\ell_\alpha$ Unit Ball $\mathbb{B}_\alpha^k$.}

Let $\mathbb{B}_\alpha^k = \{x \in \mathbb{R}^k : \norm{x}_\alpha \leq 1\}$. We seek to bound its $\epsilon$-covering number, denoted $\mathcal{N}(\mathbb{B}_\alpha^k, \norm{\cdot}_\alpha, \epsilon)$.

For $0 < \alpha < 1$, the function $\norm{\cdot}_\alpha$ is a quasi-norm on $\mathbb{R}^k$. This means it satisfies the following properties:
\begin{enumerate}
    \item $\norm{x}_\alpha \geq 0$, and $\norm{x}_\alpha = 0$ if and only if $x=0$.
    \item $\norm{c x}_\alpha = |c| \norm{x}_\alpha$ for any scalar $c \in \mathbb{R}$.
    \item There exists a constant $C_\alpha \geq 1$ such that for any $x, y \in \mathbb{R}^k$, $\norm{x+y}_\alpha \leq C_\alpha (\norm{x}_\alpha + \norm{y}_\alpha)$.
\end{enumerate}
For $\ell_\alpha$, specifically, the constant $C_\alpha = 2^{1/\alpha - 1}$.

For brevity, let $\mathcal{N}:= \mathcal{N}(\mathbb{B}_\alpha^k, \norm{\cdot}_\alpha, \epsilon)$. By definition of an $\epsilon$-net, there exist $\mathcal{N}$ points $\{x_1, \ldots, x_\mathcal{N}\} \subset \mathbb{B}_\alpha^k$ such that for any $x \in \mathbb{B}_\alpha^k$, there is some $x_j$ with $\norm{x - x_j}_\alpha \leq \epsilon$.

Consider the collection of open balls $\{B(x_j, \frac{\epsilon}{2C_\alpha})\}_{j=1}^\mathcal{N}$, where $B(x, r) = \{y \in \mathbb{R}^k : \norm{y-x}_\alpha < r\}$.

\textit{Sub-step 1.1: Disjointness of the balls.}
We claim that these balls are disjoint. Suppose, for contradiction, that $z \in B(x_i, \frac{\epsilon}{2C_\alpha}) \cap B(x_j, \frac{\epsilon}{2C_\alpha})$ for distinct $i, j \in \{1, \ldots, \mathcal{N}\}$.
Then $\norm{z - x_i}_\alpha < \frac{\epsilon}{2C_\alpha}$ and $\norm{z - x_j}_\alpha < \frac{\epsilon}{2C_\alpha}$.
Using the quasi-triangle inequality (property 3 above):
\[ \norm{x_i - x_j}_\alpha = \norm{(x_i - z) + (z - x_j)}_\alpha \leq C_\alpha (\norm{x_i - z}_\alpha + \norm{z - x_j}_\alpha), \] and so
\[ \norm{x_i - x_j}_\alpha < C_\alpha \left(\frac{\epsilon}{2C_\alpha} + \frac{\epsilon}{2C_\alpha}\right) = C_\alpha \left(\frac{\epsilon}{C_\alpha}\right) = \epsilon. \]
However, the points $\{x_1, \ldots, x_\mathcal{N}\}$ forming an $\epsilon$-net (specifically, a maximal $\epsilon$-separated set, whose size equals the covering number) must satisfy $\norm{x_i - x_j}_\alpha \geq \epsilon$ for $i \neq j$. This leads to a contradiction.
Therefore, the balls $\{B(x_j, \frac{\epsilon}{2C_\alpha})\}$ are disjoint.

\textit{Sub-step 1.2: Containment within a larger ball.}
Each $x_j \in \mathbb{B}_\alpha^k$, so $\norm{x_j}_\alpha \leq 1$.
Let $y \in B(x_j, \frac{\epsilon}{2C_\alpha})$. Then $\norm{y - x_j}_\alpha < \frac{\epsilon}{2C_\alpha}$.
Again, using the quasi-triangle inequality:
\begin{eqnarray*}
&&\norm{y}_\alpha = \norm{(y - x_j) + x_j}_\alpha \leq C_\alpha (\norm{y - x_j}_\alpha + \norm{x_j}_\alpha) \\
&\implies&\norm{y}_\alpha < C_\alpha \left(\frac{\epsilon}{2C_\alpha} + 1\right) = \frac{\epsilon}{2} + C_\alpha. 
\end{eqnarray*}
Thus, all $\mathcal{N}$ disjoint balls $B(x_j, \frac{\epsilon}{2C_\alpha})$ are contained within the open ball $B(0, C_\alpha + \frac{\epsilon}{2}) = \{z \in \mathbb{R}^k : \norm{z}_\alpha < C_\alpha + \frac{\epsilon}{2}\}$.

\textit{Sub-step 1.3: Volume comparison.}
Let $\text{Vol}(S)$ denote the Lebesgue measure (volume) of a set $S \subset \mathbb{R}^k$.
For an $\ell_\alpha$ ball centered at the origin, its volume is proportional to the $k$-th power of its radius: $\text{Vol}(B(0, r)) = r^k \cdot \text{Vol}(B(0, 1))$.
Since the $\mathcal{N}$ disjoint balls $B(x_j, \frac{\epsilon}{2C_\alpha})$ are contained in $B(0, C_\alpha + \frac{\epsilon}{2})$, their combined volume must be less than or equal to the volume of the containing ball:
\begin{eqnarray*}
&&\sum_{j=1}^\mathcal{N} \text{Vol}\left(B\left(x_j, \frac{\epsilon}{2C_\alpha}\right)\right) \leq \text{Vol}\left(B\left(0, C_\alpha + \frac{\epsilon}{2}\right)\right) \\
&\implies& \mathcal{N} \cdot \left(\frac{\epsilon}{2C_\alpha}\right)^k \text{Vol}(B(0, 1)) \leq \left(C_\alpha + \frac{\epsilon}{2}\right)^k \text{Vol}(B(0, 1)). 
\end{eqnarray*}
Since $\text{Vol}(B(0, 1)) > 0$,
\[ \mathcal{N} \cdot \left(\frac{\epsilon}{2C_\alpha}\right)^k \leq \left(C_\alpha + \frac{\epsilon}{2}\right)^k \implies \mathcal{N} \leq \left(\frac{C_\alpha + \frac{\epsilon}{2}}{\frac{\epsilon}{2C_\alpha}}\right)^k = \left(\frac{2C_\alpha^2 + C_\alpha\epsilon}{\epsilon}\right)^k = \left(\frac{2C_\alpha^2}{\epsilon} + C_\alpha\right)^k. \]
Since $0 < \epsilon < 1$, we have $C_\alpha < \frac{C_\alpha}{\epsilon}$. Therefore, $C_\alpha + \frac{2C_\alpha^2}{\epsilon} \le \frac{2C_\alpha^2 + C_\alpha}{\epsilon}$.
Let $C'_\alpha = 2C_\alpha^2 + C_\alpha$. 
Thus, we have bounded the covering number of $\mathbb{B}_\alpha^k$:
\[ \mathcal{N}(\mathbb{B}_\alpha^k, \norm{\cdot}_\alpha, \epsilon) \leq \left(\frac{C'_\alpha}{\epsilon}\right)^k. \]

\textbf{Step 2: Extending to the Sparse $\ell_\alpha$ Ball $\mathbb{B}_{\alpha,k}^N$.}

The set $\mathbb{B}_{\alpha,k}^N$ contains all vectors in $\mathbb{R}^N$ with at most $k$ non-zero entries and $\ell_\alpha$-norm at most 1. We can partition this set based on the support of its vectors.
For any subset of indices $T \subseteq \{1,2,\dots, N\}$ with $|T|=k$, let $S_T = \{x \in \mathbb{R}^N : \text{supp}(x) \subseteq T, \norm{x}_\alpha \leq 1\}$.
Then $\mathbb{B}_{\alpha,k}^N = \bigcup_{|T|=k} S_T$.
The number of such support sets $T$ is given by the binomial coefficient ${N \choose k}$.

For each fixed support $T$, the problem of finding an $\epsilon$-net for $S_T$ is equivalent to finding an $\epsilon$-net for the $k$-dimensional $\ell_\alpha$ unit ball $\mathbb{B}_\alpha^k$. This is because any vector $x \in S_T$ can be uniquely identified with a vector in $\mathbb{R}^k$ by restricting to its non-zero coordinates and ignoring their original positions in $\mathbb{R}^N$. This identification preserves the $\ell_\alpha$ quasi-norm.
Therefore, from Step 1:
\[ \mathcal{N}(S_T, \norm{\cdot}_\alpha, \epsilon) \leq \mathcal{N}(\mathbb{B}_\alpha^k, \norm{\cdot}_\alpha, \epsilon) \leq \left(\frac{C'_\alpha}{\epsilon}\right)^k. \]

To obtain an $\epsilon$-net for the entire $\mathbb{B}_{\alpha,k}^N$, we can simply take the union of the $\epsilon$-nets constructed for each $S_T$. The size of this combined net will be at most the sum of the sizes of the individual nets:
\[ \mathcal{N}(\mathbb{B}_{\alpha,k}^N, \norm{\cdot}_\alpha, \epsilon) \leq \sum_{|T|=k} \mathcal{N}(S_T, \norm{\cdot}_\alpha, \epsilon), \] which gives
\[ \mathcal{N}(\mathbb{B}_{\alpha,k}^N, \norm{\cdot}_\alpha, \epsilon) \leq {N \choose k} \cdot \left(\frac{C'_\alpha}{\epsilon}\right)^k. \]

Let $C_{\text{net}}:= C'_\alpha = 2(2^{1/\alpha-1})^2 + 2^{1/\alpha-1}$. This constant $C_{\text{net}}$ depends only on $\alpha$ and is positive.
Thus, the lemma is proven.
\end{proof}

\begin{remark}
The constant $C_{\text{net}} = 2(2^{1/\alpha-1})^2 + 2^{1/\alpha-1}$ can be quite large as $\alpha \to 0^+$. More refined bounds often involve tighter volume estimates or specific constructions that may be more complex for quasi-normed spaces. 
\end{remark}

\section{Proof of RQIP Theorem}
With the background work above, we state and prove:
\begin{theorem}[RQIP Theorem]
Let $\Omega \in \mathbb{R}^{M \times N}$ be a random matrix whose entries $\omega_{ij}$ are i.i.d. symmetric $\alpha$-stable random variables, denoted $S\alpha S(\gamma)$, with $\alpha \in (0,1)$ and scale parameter $\gamma > 0$. For any $p \in (0, \alpha)$ and any specified $k \in \{1, \dots, N\}$, for any target $\delta \in (0,1)$ and $\eta \in (0,1)$, and with constants $C_\text{con}, c_\text{con}$ as in Lemma \ref{lem:concentration} and $C_\text{net}$ as in Lemma \ref{lem:covering}, if the number of rows $M$ satisfies
\[M \geq \left( \frac{2 C_{\text{con}}}{\eta} C_{\text{net}}^k \left(\frac{3}{\delta}\right)^{k/p} \left(\frac{eN}{k}\right)^k \right)^{1/c_{\text{con}}}, \]
then $\Omega$ satisfies the $(\delta, k)$-RQIP with probability at least $1-\eta$. 
\end{theorem}
\begin{proof}
\textbf{Step 1: Concentration for a Fixed Sparse Vector}\\
Let $x$ be a fixed $k$-sparse vector with $\norm{x}_\alpha = 1$. This normalization does not affect the property, as the RQIP condition is homogeneous with respect to $\norm{x}_\alpha$.
Define $y_i = \omega_i^T x$ for $i=1, \dots, M$. Since $\omega_{ij} \sim S\alpha S(\gamma)$, each $y_i$ is an $S\alpha S(\gamma \norm{x}_\alpha)$ random variable. With $\norm{x}_\alpha = 1$, we have $y_i \sim S\alpha S(\gamma)$. The $p$-th absolute moment of $y_i$ is $\mathbb{E}[|y_i|^p] = C_{\alpha,p} (\gamma \norm{x}_\alpha)^p = C_{\alpha,p} \gamma^p$.
Let $S_x = \sum_{i=1}^M |\omega_i^T x|^p = \norm{\Omega x}_p^p$. By Lemma \ref{lem:concentration}, for a fixed $x$, the probability of deviation from the mean is bounded for all large $M>M_0(\varepsilon',\alpha,p,\gamma)$:
\[ \mathbb{P}\left( \left| \frac{1}{M} S_x - C_{\alpha,p} \gamma^p \right| > \varepsilon' \right) \leq C_{\text{con}} \cdot M^{-c_{\text{con}}}. \]
To satisfy the RQIP for this fixed $x$, we need the deviation to be bounded by $\delta_A C_{\alpha,p} \gamma^p$. Thus, we set $\varepsilon' = \delta_A C_{\alpha,p} \gamma^p$.
The probability that RQIP fails for a fixed $x$ is:
\[ \mathbb{P}\left( \left| \frac{1}{M} \norm{\Omega x}_p^p - C_{\alpha,p} \gamma^p \right| > \delta_A C_{\alpha,p} \gamma^p \right) \leq C_{\text{con}} \cdot M^{-c_{\text{con}}}. \]

\textbf{Step 2: Extension to an $\epsilon$-Net}\\
To ensure the RQIP holds for all $k$-sparse vectors, we employ an $\epsilon$-net argument. Let $\mathbb{B}_{\alpha,k}^N(1) := \{x \in \mathbb{R}^N : \norm{x}_\alpha = 1, \norm{x}_0 \le k\}$.
By Lemma \ref{lem:covering}, for a chosen $\epsilon > 0$, there exists an $\epsilon$-net $\mathcal{N}$ for $\mathbb{B}_{\alpha,k}^N(1)$ with cardinality bounded by:
\[ |\mathcal{N}| \leq \left(\frac{C_{\text{net}}}{\epsilon}\right)^k {N \choose k}. \]
We apply the union bound over all points in this net. The probability that RQIP fails for at least one $x_0 \in \mathcal{N}$ is:
\begin{align*}
\mathbb{P}(\exists x_0 \in \mathcal{N} : \text{RQIP fails for } x_0) &\leq \sum_{x_0 \in \mathcal{N}} \mathbb{P}\left( \left| \frac{1}{M} \norm{\Omega x_0}_p^p - C_{\alpha,p} \gamma^p \right| > \delta_A C_{\alpha,p} \gamma^p \right) \\
&\leq |\mathcal{N}| \cdot C_{\text{con}} \cdot M^{-c_{\text{con}}} \\
&\leq \left(\frac{C_{\text{net}}}{\epsilon}\right)^k {N \choose k} \cdot C_{\text{con}} \cdot M^{-c_{\text{con}}}.
\end{align*}
Let this probability be $\eta_1$. For RQIP to hold for all $x_0 \in \mathcal{N}$ with high probability (e.g., $1-\eta_1$), we must choose $M$ such that:
\[ M \geq \left( \frac{C_{\text{con}}}{\eta_1} \left(\frac{C_{\text{net}}}{\epsilon}\right)^k {N \choose k} \right)^{1/c_{\text{con}}}. \]
Let $E_{\mathcal{N}}$ be the event that RQIP holds for all $x_0 \in \mathcal{N}$. This event occurs with probability at least $1-\eta_1$. Note that our choice of $M$ here is $>>M_0(\delta_A,\alpha,p,\gamma).$\\

\textbf{Step 3: Extension from Net to All Sparse Vectors}\\
Now, consider an arbitrary $k$-sparse vector $x \in \mathbb{R}^N$. Without loss of generality, we assume $\norm{x}_\alpha = 1$.
By the definition of an $\epsilon$-net $\mathcal{N}$ for $\mathbb{B}_{\alpha,k}^N(1)$, there exists an $x_0 \in \mathcal{N}$ such that $\norm{x - x_0}_\alpha \leq \epsilon$.
We want to bound the total deviation $\left| \frac{1}{M} \norm{\Omega x}_p^p - C_{\alpha,p} \gamma^p \right|$.
Using the triangle inequality:
\[ \left| \frac{1}{M} \norm{\Omega x}_p^p - C_{\alpha,p} \gamma^p \right| \le \left| \frac{1}{M} \norm{\Omega x}_p^p - \frac{1}{M} \norm{\Omega x_0}_p^p \right| + \left| \frac{1}{M} \norm{\Omega x_0}_p^p - C_{\alpha,p} \gamma^p \right|. \]
On the event $E_{\mathcal{N}}$, the second term on the right hand side is bounded by $\delta_A C_{\alpha,p} \gamma^p$.

For the first term, let $z = x - x_0$. Since $x$ is $k$-sparse and $x_0$ is $k$-sparse, $z$ is $2k$-sparse. More importantly for us, $\norm{z}_\alpha \leq \epsilon$.
For $p \in (0,1)$, we have the quasi-triangle inequality: $\abs{a+b}^p \le \abs{a}^p + \abs{b}^p$ for $a,b \in \mathbb{R}$.
Applying this elementwise, we obtain: $\left| \norm{\Omega x}_p^p - \norm{\Omega x_0}_p^p \right| \le \norm{\Omega z}_p^p$.
So, we need to bound $\frac{1}{M} \norm{\Omega z}_p^p$. \\
We have:
\begin{equation}
\label{eq:inter}
\left| \frac{1}{M} \norm{\Omega x}_p^p - C_{\alpha,p} \gamma^p \right| \le \left| \frac{1}{M} \norm{\Omega x_0}_p^p - C_{\alpha,p} \gamma^p \right| + \frac{1}{M} \norm{\Omega z}_p^p.
\end{equation}
Starting from an arbitrary $x$, we get a specific $z=x-x_0$ for that $x$ using the net $\mathcal{N}$. For this particular $z$, Lemma \ref{lem:concentration} states that, for $M > M_0(\delta_A,\alpha,p,\gamma)$:
\[ \mathbb{P}\left( \left| \frac{1}{M} \norm{\Omega z}_p^p - C_{\alpha,p} (\gamma \norm{z}_\alpha)^p \right| > \delta_A C_{\alpha,p} \gamma^p \right) \leq C_{\text{con}} \cdot M^{-c_{\text{con}}}. \]
Since $\norm{z}_\alpha \le \epsilon$, we can bound the expectation: $C_{\alpha,p} (\gamma \norm{z}_\alpha)^p \le C_{\alpha,p} (\gamma \epsilon)^p$.
Combining these:
\[ \mathbb{P}\left( \frac{1}{M} \norm{\Omega z}_p^p > C_{\alpha,p} (\gamma \epsilon)^p + \delta_A C_{\alpha,p} \gamma^p \right) \leq C_{\text{con}} \cdot M^{-c_{\text{con}}}. \]
Let $\eta_2 = C_{\text{con}} \cdot M^{-c_{\text{con}}}$.\\
The total probability of failure for the entire argument is at most $\eta = \eta_1 + \eta_2$. Note that, since $M$ is chosen to satisfy the bounds for $\eta_1,$ the simpler condition for $\eta_2$ is satisfied.
With probability $1-\eta_2$, it holds for this specific $z$ that
$$ \frac{1}{M} \norm{\Omega z}_p^p \le C_{\alpha,p} (\gamma \epsilon)^p + \delta_A C_{\alpha,p} \gamma^p. $$
With the choice $C_{\alpha,p} (\gamma \epsilon)^p = \delta_A C_{\alpha,p} \gamma^p$, we obtain:
\[ \frac{1}{M} \norm{\Omega z}_p^p \le 2\delta_A C_{\alpha,p} \gamma^p. \]
If $M$ is chosen to satisfy the condition from the first net, that is, on the event $E_{\mathcal{N}}$, we have from \eqref{eq:inter} that, for any arbitrary $k$-sparse $x$,
\[
\left| \frac{1}{M} \norm{\Omega x}_p^p - C_{\alpha,p} \gamma^p \right| \le 3\delta_A C_{\alpha,p} \gamma^p.\] 
\textbf{Step 4: Determining Sample Complexity $M$}\\
The required number of measurements $M$ must satisfy two conditions, one from the net argument, and one for the final residual. Recall again that the choice of $M$ for $\eta_2 = C_{\text{con}} \cdot M^{-c_{\text{con}}}$, is subsumed by the larger $M$ driven by the net bound. We set $\eta_1=\eta_2=\eta/2, \delta_A=\delta/3$, and $\epsilon=\delta_A^{1/p}=(\delta/3)^{1/p}$, so as to satisfy the confidence and RQIP deviation requirements.\\ 
Substituting $\epsilon = (\frac{\delta}{3})^{1/p}$, we get that as long as
\[
M \geq \left( \frac{C_{\text{con}}}{\eta/2} \left(\frac{C_{\text{net}}}{(\frac{\delta}{3})^{1/p}}\right)^k {N \choose k} \right)^{1/c_{\text{con}}} = \left( \frac{2 C_{\text{con}}}{\eta} C_{\text{net}}^k \left(\frac{3}{\delta}\right)^{k/p} {N \choose k} \right)^{1/c_{\text{con}}},\]
the random matrix $\Omega$ satisfies the $(\delta, k)$-RQIP with probability at least $1-\eta$.
The exact claim of the theorem follows on recalling the elementary inequality
$
{N\choose k}\leq (\frac{eN}{k})^k.
$
\end{proof}

Note that the sample complexity depends on the relative magnitudes of $c_{\text{con}}$ and $k$. For CS applications, the goal is to have $M \ll N$. If $c_{\text{con}}$ is sufficiently small in relation to $k$, terms of the form $(\frac{N}{k})^{k/c_{\text{con}}}$ will lead to $M$ being comparable to or even greater than $N$, thereby undermining the purpose of CS. To ensure $M < N$, it would ideally be desirable for $c_{\text{con}}$ to be large enough relative to $k$ or a multiple of it. This implies a need for $\alpha/p$ to be large which, for a fixed $\alpha \in (0,1)$, suggests $p$ could be very small (close to $0$). However, choosing $p$ too small might introduce numerical instabilities or practical limitations in recovery.

Comparing this result to common RIP bounds for sub-Gaussian matrices \cite{chen2024} or to that obtained in \cite{otero2011generalized} for $S\alpha S$ random matrices with $1\leq\alpha\leq 2$, which scale as $O(k \log(N/k))$, the dependency here involves a polynomial factor in $N$. While this seems pessimistic, it may be inherently linked to the heavy-tailed nature of $\alpha$-stable distributions with $\alpha < 1$. Thus, it is likely that the dependence on $N$ cannot be improved within the RIP framework for this class of matrices, unless an improved Fuk-Nagaev type concentration holds.

Nevertheless, this theoretical result is significant because it addresses a previously less explored domain of random matrices with heavy-tailed, $S\alpha S$ entries ($\alpha<1$), which are relevant for signals corrupted by impulsive noise or extreme events. The existence of such a rigorous proof of RQIP, particularly given the challenges posed by quasinorms ($0 < \alpha < 1$) and heavy tails, is a notable contribution. It highlights the complexities and potential limitations of directly extending the RIP framework to these situations, suggesting that alternative recovery guarantees, such as those based on the NSP framework, might be necessary for practical CS with extremely heavy-tailed noise.

\providecommand{\bysame}{\leavevmode\hbox to3em{\hrulefill}\thinspace}
\providecommand{\MR}{\relax\ifhmode\unskip\space\fi MR }
\providecommand{\MRhref}[2]{%
  \href{http://www.ams.org/mathscinet-getitem?mr=#1}{#2}
}
\providecommand{\href}[2]{#2}

\begin{acks}
We are grateful to Robert Adler, Yogeshwaran Dhandapani, and Manjunath Krishnapur, for advice and encouragement.
\end{acks}


\end{document}